\documentclass{siamltex}

\usepackage{amsfonts, amsmath, amssymb, psfrag, graphicx, bbm, mathrsfs}
\usepackage{mathabx}   

\usepackage[colorlinks=true]{hyperref}
\hypersetup{urlcolor=blue, citecolor=blue, linkcolor=blue}

\newtheorem{result}{Theorem}

\def\fin{\ifmmode{\Large$\diamond$}\else{\unskip\nobreak\hfil
    \penalty50\hskip1em\null\nobreak\hfil{\Large$\diamond$}
    \parfillskip=0pt\finalhyphendemerits=0\endgraf}\fi}

\def\be#1#2\ee{\begin{equation}\label{eq:#1}#2\end{equation}}
\def\req#1{{\rm(\ref{eq:#1})}}
\def\bdm  {\begin{displaymath}}
  \def\edm  {\end{displaymath}}
\def\bdmal{\begin{displaymath}\begin{aligned}}
    \def\edmal{\end{aligned}\end{displaymath}}

\mathchardef\PhiG="0108
\mathchardef\Omega="010A
\mathchardef\Sigma="0106
\mathchardef\Gamma="0100
\mathchardef\Lambda="0103

\newcommand{\N}{{\mathord{\mathbb N}}}
\newcommand{\R}{{\mathord{\mathbb R}}}

\renewcommand{\S}{\mathord{\cal S}_{\partial\D}}
\newcommand{\Sp}{\mathord{\cal S}_{\partial\D'}}

\newcommand{\dd}{{\mathbf{d}}}

\newcommand{\xx}{{\mathbf{x}}}

\newcommand{\B}{{\cal B}}

\newcommand{\norm}[1]{\|#1\|}

\newcommand{\Span}{{\rm{span}}}
\newcommand{\rmd}{\,\mathrm{d}}
\newcommand{\ds}{\rmd s}
\newcommand{\dt}{\rmd t}

\newcommand{\eps}{\varepsilon}

\def\req#1{{\rm(\ref{eq:#1})}}

\makeatletter
\newcommand{\dupdots}{\mathinner{\mkern1mu\raise\p@
    \vbox{\kern7\p@\hbox{.}}\mkern2mu
    \raise4\p@\hbox{.}\mkern2mu\raise7\p@\hbox{.}\mkern1mu}}
\makeatother
\makeatletter
\newcount\c@MaxMatrixCols \c@MaxMatrixCols=10
\newenvironment{cmatrixc}{%
  \hskip -\arraycolsep\array{*\c@MaxMatrixCols c}%
}{%
  \endarray \hskip -\arraycolsep
}
\makeatother
\newenvironment{cmatrix}{\left[\cmatrixc}{\endmatrix\right]}

\def\thsp{\hspace*{0.1ex}}

 {\endMakeFramed}

\newcommand{\U}{{{\mathcal U}}}

\newcommand{\X}{{\mathcal X}}

\newcommand{\D}{{\mathscr D}}

\newcommand{\A}{{\cal A}}

\makeatletter
\renewcommand\@biblabel[1]{#1.}
\makeatother

\title{Lipschitz stability of an inverse conductivity problem
with two Cauchy data pairs}
\author{Martin Hanke\thanks{Institut f\"ur Mathematik, Johannes
    Gutenberg-Universit\"at Mainz, 55099 Mainz, Germany
    ({\tt hanke@math.uni-mainz.de}).}}

\begin{document}
\sloppy
\maketitle

\begin{abstract}
In 1996 Seo proved that two appropriate pairs of current and voltage data
measured on the surface of a planar homogeneous object are sufficient to 
determine a conductive polygonal inclusion with known deviating conductivity. 
Here we show that the corresponding linearized forward map is injective,
and from this we deduce Lipschitz stability of the solution of the
original nonlinear inverse problem. We also treat the case of
an insulating polygonal inclusion, in which case a single pair of Cauchy data
is already sufficient for the same purpose.
\end{abstract}


\begin{AMS}
{\sc 35R30, 35J25, 65J22}
\end{AMS}


%

\section{Introduction}
\label{Sec:Intro}
We consider the boundary value problem
\be{conductivity-equation}
   \nabla\cdot (\sigma \nabla u) \,=\, 0  \quad \text{in $\Omega$}\,, \qquad
   \frac{\partial}{\partial\nu} u \,=\, f \quad \text{on $\partial\Omega$}\,,
\ee
for the electric potential $u$ in a planar object $\Omega$, 
when a (quasi-)static boundary current $f$ with vanishing mean
is applied on its boundary. To be specific we focus on the situation 
that the object contains a so-called inclusion $\D$ of some other conducting 
material, such that the spatial conductivity distribution is given by
\be{sigma-Beretta}
   \sigma \,=\, \begin{cases}
                   k & \text{in $\D$}\,, \\
                   1 & \text{in $\Omega\setminus\overline\D$}\,,
                \end{cases}
\ee
with a nonnegative value $k\neq 1$.

The inverse conductivity problem that we are interested in 
seeks to recover the inclusion
(i.e., location and shape of $\D$) from measurements of the potential on
the boundary of the object.
The question whether this inverse problem is uniquely solvable
has a long-standing history. In the formal (degenerate) case $k=0$, i.e., 
when the inclusion is taken to be insulating
(cf.\ the discussion in Sect.~\ref{Sec:insulating_case}) and its complement
is a connected set, then it is known that
the Cauchy data of $u$ on $\partial\Omega$ do indeed uniquely determine 
the inclusion; for a proof of this result, cf., e.g., 
Beretta and Vessella~\cite{BeVe98}.

Less is known, however, when the conductivity in the inclusion is a 
nonzero constant different from one. For this case
Friedman and Isakov~\cite{FrIs89} proved that a convex polygonal inclusion
is uniquely determined by this data, provided its conductivity
is known and the diameter of $\D$ is  
smaller than its distance to the boundary of the object.
Barcel\'o, Fabes, and Seo~\cite{BFS94}, and 
Alessandrini and Isakov~\cite{AlIs96} were able to drop the size constraint 
on the polygonal inclusion for the prize of accepting only certain admissible 
boundary currents $f$ for probing the object. 
Finally, Seo~\cite{Seo96} proved that the convexity assumption 
on the polygon can also be omitted when the boundary potentials 
for two appropriate probing currents $f_1,f_2$ are given; 
see Assumption~\req{Ass:Seo} in Section~\ref{Sec:inverse}.

It is known that inverse conductivity problems in general are badly ill-posed
in the sense that the solution lacks continuous dependence on the given data;
some conditional stability -- generically of logarithmic type -- 
can be restored by providing further a priori information on the conductivity,
cf., e.g., \cite{BBR01,BeVe98,CFR10}.

The past twenty years have seen increasing activities in deriving Lipschitz
stability estimates for inverse conductivity problems
under very restrictive conditions on the set of admissible conductivities. 
One of the first results in this direction was
obtained by Alessandrini and Vessella~\cite{AlVe05} who showed that if
\emph{all} Cauchy pairs for $u$ in \req{conductivity-equation} 
are known, i.e., if the full (or the local) Neumann-Dirichlet operator 
associated with the differential equation in \req{conductivity-equation} 
is given, then the \emph{values} of a piecewise constant conductivity 
$\sigma$ with respect to a known partitioning of $\Omega$
into finitely many subdomains depend Lipschitz continuously on these data.
Later it was proved by Harrach~\cite{Harr19} that already a finite number of 
Cauchy pairs is sufficient for this result to hold; 
see also the work by Alberti and Santacesaria~\cite{AlSa19,AlSa22}.

Whereas these results assume the spatial structure of $\sigma$ to be known
and the quantitative details are being searched for, 
the conductivity Ansatz~\req{sigma-Beretta} with known $k$ but 
unknown form and location of the inclusion, was treated by 
Beretta and Francini~\cite{BeFr22}. 
They established Lipschitz stability in terms of the Hausdorff distance 
between the boundaries of two admissible inclusions, if it is a priori 
known that they have the shape of a polygon. In contrast to the aforementioned 
uniqueness results by Beretta/Vessella and by Seo, however, the stability 
result in \cite{BeFr22} again requires the full Neumann-Dirichlet map as data. 
See also \cite{ABFV22,BFV21} for extensions of this finding to layered 
-- instead of homogeneous -- background media and to polyhedral inclusions
in three space dimensions, respectively.  

Recently, Alberti, \'{A}rroyo, and Santacesaria~\cite{AAS23} showed 
(for triangular inclusions) that the Lipschitz result by Beretta and Francini
remains valid, if the given data correspond to a sufficiently large, 
but finite number of probing currents. Here we prove the following statement
in the spirit of Seo's original uniqueness result:
If only two pairs of Cauchy data are given which fulfill 
Seo's uniqueness assumption for a conductive inclusion,
then this minimal dataset is enough to have Lipschitz stability. 
We also consider the degenerate case of an insulating polygonal inclusion
and establish Lipschitz stability for a single (nontrivial) Cauchy data pair.
We mention in passing that for a somewhat related setting, 
namely the reconstruction of a linear crack within a homogeneous planar object,
Lipschitz stability with only two pairs of Cauchy data had been
established by Alessandrini, Beretta, and Vessella~\cite{ABV96} back in 1996.  

In contrast to the analysis in the above works which, in principle,
allow for an evaluation of the corresponding Lipschitz constant, our method
is non-quantitative. Rather, we use a general methodology worked out
by Bourgeois~\cite{Bour13}, building on earlier work by
Bacchelli and Vessella~\cite{BaVe06}.
As is transparent from their results, the key ingredients for Lipschitz
stability in general nonlinear inverse problems are
\begin{enumerate}
\item a specification of the quantity of interest in terms of finitely
many parameters,
\item the restriction of these parameters to a compact set,
\item injectivity of the forward operator on this compact set,
\item continuous differentiability of the forward operator, and
\item injectivity of the Jacobian of the forward operator.
\end{enumerate}
Since we can build on the uniqueness results by Seo and Beretta/Vessella,
respectively, and since differentiability results are also available,
it remains for us to investigate the injectivity of the associated Jacobian. 

To do so we need a specification of this Jacobian (the so-called shape
derivative) in terms of an inhomogeneous transmission problem for the 
Laplace equation; see~\req{HeRu98}. This specification -- which was well-known
for inclusions with smooth boundaries -- has been verified for polygonal
inclusions in a companion paper~\cite{Hank24a} submitted elsewhere, which is
currently under peer-review. It should be emphasized that all the theorems 
in the present paper hinge upon this auxiliary result.

The outline of this paper is as follows. 
We start in Section~\ref{Sec:forward} by reviewing known properties of the
solution $u$ of \req{conductivity-equation}, when the inclusion $\D$ is a
conductive polygon, including the aforementioned differentiability result 
with respect to its shape. In Section~\ref{Sec:inverse} we specify the
associated inverse problem as it has been introduced by Seo, and we prove
that the respective Jacobian is injective. 
Then, in Section~\ref{Sec:Lipschitz}, we adapt the method 
from \cite{BaVe06,Bour13} to our needs and establish the corresponding
Lipschitz stability result (Theorem~\ref{Thm:Lipschitz}).
We conclude the paper by treating the case of an insulating polygonal 
inclusion in Section~\ref{Sec:insulating_case}.

\section{The forward problem for a conductive polygonal inclusion}
\label{Sec:forward}
We assume throughout that $\Omega$ is a two-dimensional bounded domain 
with smooth boundary, and that the inclusion $\D$ is a
polygonal domain with simply connected closure $\overline\D\subset\Omega$. 
We denote by $\nu$ the outer normal vector on the boundaries of $\D$ 
and of $\Omega$, respectively. 
Concerning the spatial conductivity distribution~\req{sigma-Beretta} 
we make the assumption that the constant conductivity 
$k\in\R^+\setminus\{1\}$ in $\D$ is known and fixed; 
see Section~\ref{Sec:insulating_case} for the case when $k=0$. 
When the probing boundary current $f$ satisfies
\bdm
   f\in L^2_\diamond(\partial\Omega) \,=\,
   \Bigl\{\,f\in L^2(\partial\Omega) \,:\,
            \int_{\partial\Omega}f\ds = 0\,\Bigr\}\,, 
\edm
then the corresponding electric potential $u$ is the unique weak solution  
\bdm
   u\in H^1_\diamond(\Omega) \,=\,
   \Bigl\{\,u \in H^1(\Omega) \,:\,
            \int_{\partial\Omega}u\ds = 0\,\Bigr\} 
\edm
of \req{conductivity-equation}.

Let $x_i$ and $\Gamma_i$, $i=1,\dots,n\geq 3$,
denote the vertices and (relatively open) edges of $\D$,
respectively, where we assume that $x_i$ connects $\Gamma_i$ and $\Gamma_{i+1}$.
Here and throughout we identify $\Gamma_{n+1}$ with $\Gamma_1$, 
and also $x_{i+n}$ with $x_i$ for $i=1,\dots,n$, respectively.
On $\Gamma_i$ we let the unit tangent vector $\tau$ point in the direction of
$x_i$. We stipulate the general assumption that the induced orientation of 
$\partial\D$ is counterclockwise,
and that the interior angles $\alpha_i\in(0,2\pi)$, $i=1,\dots,n$, 
are all different from $\pi$. A polygon which satisfies all the above
requirements will subsequently be called \emph{admissible}.

The two components $u_-=u|_\D$ and $u_+=u|_{\Omega\setminus\overline\D}$ of $u$
are harmonic functions. Moreover, they satisfy the transmission conditions
\bdm
   u_- \,=\, u_+ \qquad \text{and} \qquad 
   k\,\frac{\partial}{\partial\nu} u_- \,=\, \frac{\partial}{\partial\nu} u_+
\edm
on every edge $\Gamma_i$. 
Therefore both components can be extended by reflection
across each of the edges of $\D$, and hence, they both are infinitely smooth
and all their derivatives extend continuously onto the edges.
Concerning the behavior of $u$ at the vertices we introduce 
a local coordinate system for
\bdm
   x \in \B_{r_0}(x_i) \,=\, \bigl\{x\in\R^2\,:\, |x-x_i|<r_0\bigr\}\,,
\edm
with $r_0$ sufficiently small, namely
\be{localcoordinates}
   x \,=\, x_i \,+\, 
           \bigl(r\cos(\theta_i+\theta),r\sin(\theta_i+\theta)\bigr)\,, \qquad
   0 < r < r_0\,, \ 0 \leq \theta < 2\pi\,,
\ee
where $\theta_i$ is such that the values
$\theta=0$, $\theta\in(0,\alpha_i)$, $\theta=\alpha_i$, and 
$\theta\in(\alpha_i,2\pi)$ correspond to points on $\Gamma_{i+1}$, in $\D$, 
on $\Gamma_i$, and in $\Omega\setminus\overline\D$, respectively.
It has been shown in \cite{BFI92} that in this coordinate system
the potential $u$ has an asymptotic expansion
\be{BFI92}
   u(x) \,=\, u(x_i) 
              \,+\, \sum_{j=1}^\infty \beta_{ij} y_{ij}(\theta)r^{\gamma_{ij}}\,,
\ee
where $y_{ij}$ are continuous functions of the polar angle, given by
(in general different) nontrivial linear combinations of 
$\cos\gamma_{ij}\theta$ and $\sin\gamma_{ij}\theta$ in $(0,\alpha_i)$ and in 
$(\alpha_i,2\pi)$, 
respectively, and the exponents $\gamma_{ij}$, $j=1,2,\dots$, are all 
the positive solutions $\gamma$ of
\be{gamma-cond}
   \bigl|\thsp\sin \gamma(\alpha_i-\pi)\bigr| 
   \,=\, \lambda\thsp \bigl|\thsp\sin\gamma\pi\bigr|\,, \qquad 
   \lambda \,=\, \Bigl|\frac{k+1}{k-1}\Bigr|\,,
\ee
which we assume to be in increasing order. Both, $y_{ij}$ and $\gamma_{ij}$ 
are independent of the probing current; only the expansion coefficients 
$\beta_{ij}=\beta_{ij}[f]\in\R$ depend
linearly on $f$. See \cite{BFI92}, \cite{Seo96}, or \cite{Hank24a} for
further details.
We refer to Figure~\ref{Fig:sines} for a graphical illustration of 
equation~\req{gamma-cond}. Since the amplitude $\lambda$ of the 
sine wave on the right-hand side of \req{gamma-cond} -- the darker graph
in Figure~\ref{Fig:sines} -- is always greater than one, and
since $0<|\pi-\alpha_i|<\pi$, it is not difficult to see that
\be{gamma}
   \frac{1}{2} \,<\, \gamma_{i1} \,<\, 1\,, \qquad
   1 \,<\, \gamma_{i2} \,<\, \frac{3}{2}\,, \qquad \text{and} \qquad 
   \gamma_{ij} \,>\, \frac{3}{2} \quad \text{for $j \geq 3$}\,.
\ee

\begin{figure}
  \centerline{
    \includegraphics[width=7cm]{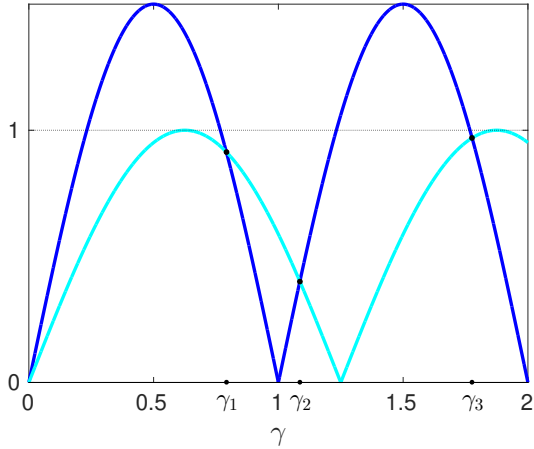}}
  \caption{The graphs of the two functions of $\gamma$ in \req{gamma-cond}:
           The lighter curve corresponds to the function on the left-hand side, 
           the darker one is the graph of the function on the right-hand side. 
           The marked abscissae of the three intersection points of the two 
           graphs are the solutions of \req{gamma-cond} in the 
           interval $(0,2]$.}
  \label{Fig:sines}
\end{figure}

We denote by
\be{Lambda}
   \Lambda_f\,:\,\D \,\mapsto \, u|_{\partial\Omega}\,,
\ee
the map, which takes an admissible polygon $\D$ onto the trace of the
solution $u$ of \req{conductivity-equation} on $\partial\Omega$. 
Let $d_i\in\R^2$, $i=1,\dots,n$, be given. 
Then we define a vector field $h:\partial\D\to\R^2$ by a piecewise 
linear interpolation of the data
\bdm
   h(x_i) \,=\, d_i\,, \qquad i=1,\dots,n\,,
\edm
i.e., both vector components of $h$ belong to the space of linear splines
over $\partial\D$ with the vertices $x_i$ as its nodes. 
If $h\to 0$ in any norm on this (finite-dimensional) linear space 
denoted by $\S^2$, e.g., with respect to the norm
\be{norm-h}
   \norm{h} \,=\, \max_{i=1,\dots,n}|h(x_i)|\,,
\ee
then it has been shown in \cite{BFV17} that the operator in \req{Lambda} is
Fr\'echet differentiable with a (shape) derivative 
$\partial\Lambda_f(\D)\in{\cal L}(\S^2,L^2_\diamond(\partial\Omega))$. 
An explicit representation of the derivative
$\partial\Lambda_f(\D)h$ in the direction of $h$ is given by the trace on
$\partial\Omega$ of the solution $u'$ of the inhomogeneous transmission problem
\be{HeRu98}
\begin{array}{c}
   \Delta u' \,=\, 0 \quad \text{in $\Omega\setminus\partial\D$}\,, \qquad
   \dfrac{\partial}{\partial\nu}u' \,=\, 0 \quad \text{on $\partial\Omega$}\,,
   \qquad 
   {\displaystyle \int_{\partial\Omega} u'\ds \,=\, 0\,,}
   \\[3ex]
   \phantom{\dfrac{\partial}{\partial\nu}\!}
   u'_+-u'_- \,=\, (1-k)(h\cdot\nu)\dfrac{\partial}{\partial\nu}u_{-} \qquad
   \text{on $\partial\D$}\,, \\[3ex]
   \dfrac{\partial}{\partial\nu} u'_+ - k\dfrac{\partial}{\partial\nu} u'_-
      \,=\, (1-k)\,\dfrac{\partial}{\partial\tau}
                \Bigl((h\cdot\nu)\dfrac{\partial}{\partial\tau}u\Bigr) \qquad
   \text{on $\partial\D$}\,,
\end{array}
\ee
where $u'_-=u'|_\D$ and $u'_+=u'|_{\Omega\setminus\overline\D}$.
This connection was first established by 
Hettlich and Rundell~\cite{HeRu98} for smooth inclusions, and 
it is shown in \cite{Hank24a} that it holds true for admissible polygons
$\D$ as well.
Take note that the inhomogeneous transmission data in \req{HeRu98} are
infinitely smooth on $\partial\D$ except for the vertices, where in general
$h\cdot\nu$ is discontinuous and the directional derivatives of $u$ tend to 
infinity; compare~\req{BFI92} and \req{gamma}.

\section{The inverse problem for a conductive polygonal inclusion}
\label{Sec:inverse}
In \cite{Seo96} Seo investigated the forward map~\req{Lambda} for two linearly 
independent piecewise continuous probing currents
$f_1,f_2\in L^2_\diamond(\partial\Omega)$ under the assumption that the set
\be{Ass:Seo}
   \bigl\{\, x\in\partial\Omega\,:\, f(x) \geq 0\,\bigr\}
\ee
is connected for every linear combination $f= \mu_1 f_1 + \mu_2 f_2$ of them;
see \cite{Seo96} for examples, how to choose $f_1$ and $f_2$ this way.
This somewhat strange assumption arises in the context of an
auxiliary result by Seo, the proof of which can also be found in \cite{Seo96}:

\begin{result}
\label{Thm:Seo}
Let $f\in L^2_\diamond(\partial\Omega)\setminus\{0\}$ be a piecewise continuous
function, and assume that the weak solution $u\in H^1_\diamond(\Omega)$ of
\req{conductivity-equation} satisfies
\bdm
   \bigl|u(x)-u(x_0)\bigr| \,\leq\, C|x-x_0|^{3/2}
\edm
for some $x_0\in\Omega$, $C>0$, and all $x\in\Omega$. Then the set
$\{x\in\partial\Omega\,:\,f(x)\geq 0\}$ is not connected. 
\end{result}

Based on Theorem~\ref{Thm:Seo} Seo showed that the corresponding operator
\be{Map:Seo}
   \Lambda_{f_1,f_2}\,:\,\D \,\mapsto\, 
   \bigl(\Lambda_{f_1}(\D),\Lambda_{f_2}(\D)\bigr)
\ee
is injective, i.e., the traces on $\partial\Omega$ of the two potentials
\req{conductivity-equation} corresponding to the boundary data $f=f_{1,2}$
uniquely determine an admissible polygonal inclusion.
Obviously, the operator $\Lambda_{f_1,f_2}$ of \req{Map:Seo} is also 
Fr\'echet differentiable with shape derivative
\bdm
   \partial\Lambda_{f_1,f_2}(\D)h \,=\,
   \bigl(\partial\Lambda_{f_1}(\D)h,\partial\Lambda_{f_2}(\D)h\bigr)
   \,\in\,{\cal L}(\S^2,(L^2_\diamond(\partial\Omega))^2)
\edm
for every $h\in\S^2$. 
In the sequel we investigate the linearized inverse problem.

\begin{theorem}
\label{Thm:Fprime-injective}
Let $\D$ be an admissible polygon,
and let the two piecewise continuous probing currents 
$f_1,f_2\in L^2_\diamond(\partial\Omega)$ satisfy Assumption~\req{Ass:Seo}.
Then $\partial\Lambda_{f_1,f_2}(\D)$ is injective.
\end{theorem}

\begin{proof}
Assume that $\partial\Lambda_{f_1,f_2}(\D)h=0$ for some nontrivial $h\in\S^2$.
Since the two normal vectors $\nu_i$ and $\nu_{i+1}$ of any two
neighboring edges $\Gamma_i$ and $\Gamma_{i+1}$ of $\D$ form a basis of $\R^2$, 
and since $h\neq 0$, there is at least one vertex $x_i$ of $\D$,
where $h(x_i)\cdot\nu_i$ or $h(x_i)\cdot\nu_{i+1}$ is different from zero. 
Without loss of generality we can assume $x_1$ to be such a vertex, and that 
\be{varphi-local}
   h(x)\cdot\nu(x)
   \,=\, \begin{cases}
            1 + b_2|x-x_1|\,, & x \in \Gamma_2\,, \\
            a + b_1|x-x_1|\,, & x \in \Gamma_1\,,
         \end{cases}
\ee
for certain real parameters $a$, $b_1$, and $b_2$.

Consider now a fixed boundary current $f\in\Span\{f_1,f_2\}$.
Then 
\be{domderiv0}
   \partial\Lambda_f(\D)h \,=\, 0\,,
\ee
and by virtue of \req{domderiv0} the associated 
potential $u'$ of \req{HeRu98} has homogeneous Cauchy data on $\partial\Omega$.
According to Holmgren's theorem this implies that $u'=0$ in all of 
$\Omega\setminus\overline\D$. It therefore follows from \req{HeRu98} that
$u'|_\D$ is a harmonic function with Cauchy data
\begin{subequations}
\label{eq:uprime-noninjective}
\be{uprime-noninjective-a}
   u' \,=\, (k-1)\thsp(h\cdot\nu)\,\frac{\partial}{\partial\nu} u_-\\
\ee
and
\be{uprime-noninjective-b}
   \frac{\partial}{\partial\nu} u'
   \,=\, \frac{k-1}{k}\, \frac{\partial}{\partial\tau}
         \Bigl((h\cdot\nu) \frac{\partial}{\partial\tau}u\Bigr)\qquad 
\ee
\end{subequations}
on every edge $\Gamma_i$.

Since the series~\req{BFI92} can be differentiated termwise and infinitely
often with respect to $r$ and $\theta\in[0,\alpha_i]$, compare~\cite{BFI92}, 
we conclude from \req{uprime-noninjective} that $u'$ and its 
Neumann derivative admit the following
series expansions for $x\in\Gamma_2$ close to $x_1$ in the associated local
coordinate system (for ease of simplicity we omit the index $i=1$ in all
terms of \req{BFI92}):
\begin{subequations}
\label{eq:Cauchy}
\begin{align}
\nonumber
   u'(x)
   &\,=\, -(k-1)(1+b_2r)\!\left.\frac{1}{r}
          \frac{\partial}{\partial\theta} u_-(x)\right|_{\theta=0}
   \,=\, (1-k)(1+b_2r)\sum_{j=1}^\infty \beta_jy_j'(0)r^{\gamma_j-1}\\[1ex]
\label{eq:Cauchy-a}
   &\,=\, (1-k)\sum_{j=1}^\infty \beta_jy_j'(0) r^{\gamma_j-1}
          \,+\, (1-k) b_2 \sum_{j=1}^\infty \beta_jy_j'(0) r^{\gamma_j}
\intertext{and}
\nonumber
   \frac{\partial}{\partial\nu} u'(x)
   &\,=\, \frac{k-1}{k}\!\left.\frac{\partial}{\partial r}
          \Bigl((1+b_2r)\frac{\partial}{\partial r} u(x)\Big)\right|_{\theta=0}
          \\[1ex]
\label{eq:Cauchy-b}
   &\,=\, \frac{k-1}{k}\,
          \sum_{j=1}^\infty \beta_j\gamma_j(\gamma_j-1)y_j(0)r^{\gamma_j-2}
          \,+\,
          \frac{k-1}{k}\,b_2
          \sum_{j=1}^\infty \beta_j\gamma_j^2y_j(0)r^{\gamma_j-1}\,.
\end{align}
\end{subequations}
Since the right-hand sides of \req{Cauchy-a} and \req{Cauchy-b} are analytic
functions of $0<r<r_0$, the local Cauchy problem~\req{uprime-noninjective}
on this portion of $\Gamma_2$ has a unique harmonic solution, 
which can be written down explicitly in the same coordinates, i.e.,
\be{uprimeCauchy}
\begin{aligned}
   u'(x)
   &\,=\, \frac{1-k}{k} \sum_{j=1}^\infty
          \Bigl(k \beta_jy_j'(0) \cos(\gamma_j-1)\theta \,+\,
                \beta_j\gamma_j y_j(0)\sin(\gamma_j-1)\theta\Bigr) r^{\gamma_j-1}
          \\[1ex]
   &\phantom{\,=\,}\ +\,
          \frac{1-k}{k} \,b_2 \sum_{j=1}^\infty
          \Bigl(k \beta_jy_j'(0) \cos\gamma_j\theta \,+\,
                \beta_j\gamma_j y_j(0)\sin\gamma_j\theta\Bigr) r^{\gamma_j}
\end{aligned}
\ee
for $0\leq\theta\leq\alpha_1$ and $0<r<r_0$;
the validity of \req{uprimeCauchy} can be checked by using the fact that
\bdm
   \left.\frac{\partial}{\partial\nu}u'(x)\right|_{\Gamma_2}
   \,=\, -\!\left.
          \frac{1}{r}\frac{\partial}{\partial\theta} u_-(x)
          \right|_{\theta=0}\,.
\edm

In particular, when $\theta=\alpha_1$ it follows from \req{gamma}
that for $x\in\Gamma_1$ we have
\bdmal
   u'(x)
   &\,=\, \frac{1-k}{k}
          \Bigl(
            k\beta_1y_1'(0)\cos(\gamma_1-1)\alpha_1 
            \,+\, \beta_1\gamma_1y_1(0)\sin(\gamma_1-1)\alpha_1
          \Bigr) r^{\gamma_1-1}\\[1ex]
   &\phantom{\,=\,}\ +\,
          \frac{1-k}{k}
          \Bigl(
            k\beta_2y_2'(0)\cos(\gamma_2-1)\alpha_1
            \,+\, \beta_2\gamma_2y_2(0)\sin(\gamma_2-1)\alpha_1
          \Bigr) r^{\gamma_2-1}\\[1.5ex]
   &\phantom{\,=\,}\ +\, O(r^{1/2})
\edmal
as $r=|x-x_1|\to 0$, 
while at the same time, according to \req{uprime-noninjective-a},
\req{varphi-local}, and \req{BFI92},
\bdmal
   u'(x)
   &\,=\, (k-1) (a+b_1r)\!\left.
           \frac{1}{r}\frac{\partial}{\partial\theta} u_-(x)
          \right|_{\theta=\alpha_1}
    \!\!=\, 
          (k-1) (a+b_1r) \sum_{j=1}^\infty\beta_jy_j'(\alpha_1) r^{\gamma_j-1}\\[1ex]
   &\,=\, (k-1) a\beta_1y_1'(\alpha_1)r^{\gamma_1-1}
          \,+\, (k-1)a\beta_2y_2'(\alpha_1)r^{\gamma_2-1}
          \,+\, O(r^{1/2})
\edmal
for the same boundary points $x\in\Gamma_1$.
A comparison of the leading order terms thus yields the two equations
\be{condition1-tmp}
   ka\beta_jy_j'(\alpha_1) 
   \,=\, -\beta_j
          \bigl(k y_j'(0)\cos(\gamma_j-1)\alpha_1
                \,+\, \gamma_jy_j(0)\sin(\gamma_j-1)\alpha_1\bigr)\,, 
\ee
$j=1,2$.
Now we recall that
\be{y12}
   y_j(\theta) \,=\, A_j\cos\gamma_j\theta \,+\, B_j\sin\gamma_j\theta
   \qquad \text{for $\theta\in [0,\alpha_1]$}
\ee
and $j=1,2$, with certain coefficients $A_j,B_j\in\R$ with $A_j^2+B_j^2\neq 0$.
Inserting this into \req{condition1-tmp} we arrive at
\be{condition1}
   k(ac_j + c_j') \beta_jB_j  \,=\, (kas_j - s_j')\beta_jA_j \,, \qquad j=1,2\,,
\ee
where we have introduced the abbreviations
\bdm
   c_j = \cos\gamma_j\alpha_1\,, \quad s_j = \sin\gamma_j\alpha_1\,, \quad
   c_j' = \cos(\gamma_j-1)\alpha_1\,, \quad
   s_j' = \sin(\gamma_j-1)\alpha_1
\edm
for general $j\in\N$.

Likewise we can use \req{uprimeCauchy} to evaluate the Neumann derivative
of $u'$ on $\Gamma_1$ near $x_1$, which gives
\begin{align}
\nonumber
   \frac{\partial}{\partial\nu} u'(x)
   &\,=\, \left.
          \frac{1}{r}\frac{\partial}{\partial\theta} u'(x)\right|_{\theta=\alpha_1}
          \\[1ex]
\nonumber
   &\,=\, \frac{1-k}{k} \sum_{j=1}^\infty (\gamma_j-1)
          \Bigl(
             \beta_j\gamma_jy_j(0)c_j' \,-\, k\beta_jy_j'(0)s_j'
          \Bigr) r^{\gamma_j-2}\\[1ex]
\nonumber
   &\qquad \qquad 
          \,+\, \frac{1-k}{k}\, b_2 \sum_{j=1}^\infty\gamma_j
          \Bigl(\beta_j\gamma_jy_j(0)c_j\,-\, k\beta_jy_j'(0)s_j\Bigr) 
          r^{\gamma_j-1}\\[1ex]
\nonumber
   &\,=\, \frac{1-k}{k}(\gamma_1-1)
           \Bigl(\beta_1\gamma_1y_1(0)c_1' \,-\, k\beta_1y_1'(0)s_1'\Bigr) 
           r^{\gamma_1-2}\\[1ex]
\nonumber
   &\qquad \qquad
           \,+\,
           \frac{1-k}{k}(\gamma_2-1)
           \Bigl(\beta_2\gamma_2y_2(0)c_2'\,-\, k\beta_2y_2'(0)s_2'\Bigr) 
           r^{\gamma_2-2} \,+\, O(r^{-1/2})
\intertext{for $r\to 0$, and compare this with \req{uprime-noninjective-b}:}
\nonumber
   \frac{\partial}{\partial\nu} u'(x)
   &\,=\, \frac{k-1}{k}
          \left.
          \frac{\partial}{\partial r}
             \Bigl((a+b_1r)\frac{\partial}{\partial r} u(x)\Bigr)
          \right|_{\theta=\alpha_1}\\[1ex]
\nonumber
   &\,=\, \frac{k-1}{k}\,\gamma_1(\gamma_1-1)a\beta_1y_1(\alpha_1)r^{\gamma_1-2}
          \,+\, \frac{k-1}{k}\,\gamma_2(\gamma_2-1)a\beta_2y_2(\alpha_1)
                r^{\gamma_2-2}\\[1ex]
\nonumber
   &\phantom{\,=\,}\ +\, O(r^{-1/2})\,.
\end{align}
Inserting~\req{y12} we thus obtain a second pair of equations,
\be{condition2}
   (ks'_j - as_j)\beta_jB_j \,=\, (ac_j+c'_j)\beta_jA_j\,, \qquad j=1,2\,.
\ee

The four equations in \req{condition1}, \req{condition2} can be rearranged
in two homogeneous linear systems
\be{LGS}
   M_j \begin{cmatrix} \beta_j A_j \\ \beta_j B_j \end{cmatrix}
   \,=\,
   \begin{cmatrix}
      kas_j - s_j' & -k(ac_j + c_j')\\
      ac_j + c_j' & as_j - ks_j'
   \end{cmatrix}
   \begin{cmatrix} \beta_j A_j \\ \beta_j B_j \end{cmatrix}
   \,=\,
   \begin{cmatrix} 0 \\ 0 \end{cmatrix}, \qquad j=1,2\,.
\ee
As mentioned before, the entries of the two matrices $M_j$
only depend on the geometry of the problem and not on the probing current. 
The probing current $f$ only enters into \req{LGS} via the coefficients 
$\beta_1=\beta_1[f]$ and $\beta_2=\beta_2[f]$. Further, since 
$A_j^2+B_j^2 \neq 0$ for $j=1,2$, it follows that $\beta_j[f]=0$ for 
\emph{every} probing current $f\in\Span\{f_1,f_2\}$,
if the matrix $M_j$ happens to be nonsingular.

Let us therefore make the assumption that both matrices $M_1$ and $M_2$ 
are singular. Then we must have
\bdmal
   0 &\,=\, (kas_j-s_j')(as_j-ks_j') \,+\, k(ac_j+c_j')^2\\[1ex]
     &\,=\, k(1+2ac_jc_j'+a^2) \,-\, (k^2+1)as_js_j'\,, \qquad j=1,2\,,
\edmal
because $c_j^2+s_j^2=c_j'{}^2+s_j'{}^2=1$. Since
\bdmal
   c_jc_j'
   &\,=\, c_j\cos(\gamma_j-1)\alpha_1
    \,=\ c_j(\cos\gamma_j\alpha_1\cos\alpha_1
             \,+\, \sin\gamma_j\alpha_1\sin\alpha_1)\\[1ex]
   &\,=\, c_j^2\cos\alpha_1 \,+\, c_js_j\sin\alpha_1
    \,=\, \cos\alpha_1 \,-\, s_j(s_j\cos\alpha_1 - c_j\sin\alpha_1)\\[1ex]          
   &\,=\, \cos\alpha_1 \,-\, s_j\sin(\gamma_j-1)\alpha_1
    \,=\, \cos\alpha_1 \,-\, s_js_j'\,, 
\edmal
the previous equations can be rewritten as
\be{aquad}
   0 \,=\, k(1+2a\cos\alpha_1+a^2) \,-\, (k+1)^2as_js_j'\,, \qquad j=1,2\,.
\ee
From this we immediately deduce that $a$ must be different from zero 
in this case. Accordingly, as \req{aquad} is bound to hold for $j=1$ and 
$j=2$ simultaneously, we can subtract these two equations, and conclude that
\be{s1s1s2s2}
   s_1s_1' \,=\, s_2s_2'\,.
\ee

To obtain a contradiction we turn to \req{gamma-cond} and 
Figure~\ref{Fig:sines} and distinguish two cases.
If $0<\alpha_1<\pi$, then the appropriate instances of \req{gamma-cond} take 
the form
\bdm
   \sin\gamma_1(\pi-\alpha_1) \,=\, \lambda\sin\gamma_1\pi\,,
   \qquad
   \sin\gamma_2(\pi-\alpha_1) \,=\, -\lambda\sin\gamma_2\pi\,,
\edm
which can be rewritten with the help of the angle sum formula as
\begin{subequations}
\label{eq:s1s2-i}
\begin{align}
\label{eq:s1-i}
   (c_1-\lambda)\sin\gamma_1\pi &\,=\, s_1\cos\gamma_1\pi\,, \\
\label{eq:s2-i}
   (c_2+\lambda)\sin\gamma_2\pi &\,=\, s_2\cos\gamma_2\pi\,.
\end{align}
\end{subequations}
Since $\lambda>1$ and $\pi/2<\gamma_1\pi<\pi$ by virtue of \req{gamma}, 
the left-hand side of \req{s1-i} is negative, and so is $\cos\gamma_1\pi$.
Likewise, since $\pi<\gamma_2\pi<3\pi/2$, the left-hand side
of \req{s2-i} is negative, and again, so is $\cos\gamma_2\pi$. Therefore,
we conclude from \req{s1s2-i} that
\be{s1s2sign}
   s_1>0 \qquad \text{and} \qquad s_2>0\,.
\ee
On the other hand, $-1/2<\gamma_1-1<0$, and therefore
$(\gamma_1-1)\alpha_1\in(-\pi/2,0)$ in this first case. This shows that
\bdm
   s_1' \,=\, \sin(\gamma_1-1)\alpha_1 \,<\, 0\,,
\edm
whereas
\bdm
   s_2' \,=\, \sin(\gamma_2-1)\alpha_1 \,>\, 0\,,
\edm
because $0<\gamma_2-1<1/2$.
Together with \req{s1s2sign} this contradicts \req{s1s1s2s2} in the case,
where $0<\alpha_1<\pi$.

In the other case, where $\pi<\alpha_1<2\pi$, \req{gamma-cond} implies that
\bdm
   \sin\gamma_1(\alpha_1-\pi) \,=\, \lambda\sin\gamma_1\pi\,,
   \qquad
   \sin\gamma_2(\alpha_1-\pi) \,=\, -\lambda\sin\gamma_2\pi\,,
\edm
and this yields
\bdmal
   (c_1+\lambda)\sin\gamma_1\pi &\,=\, s_1\cos\gamma_1\pi\,, \\
   (c_2-\lambda)\sin\gamma_2\pi &\,=\, s_2\cos\gamma_2\pi\,.
\edmal
This shows that
\bdm
   \phantom{'}s_1 < 0 \qquad \text{and} \qquad s_2 < 0\,,
\edm
while
\bdm
   s_1' < 0 \qquad \text{and} \qquad s_2'>0
\edm
in this case, because $-\pi<(\gamma_1-1)\alpha_1<0$ and 
$0<(\gamma_2-1)\alpha_1<\pi$, which again contradicts \req{s1s1s2s2}.

We thus have brought our assumption, that both matrices $M_1$ and $M_2$
are singular, to a contradiction. But if $M_1$ is nonsingular, then
\begin{align}
\nonumber
   \beta_1[\mu_1f_1+\mu_2f_2] &\,=\, 0 \qquad 
   \text{for every $\mu_1,\mu_2\in\R$}\,,
\intertext{while we can enforce}
\nonumber
   \beta_2[\mu_1f_1+\mu_2f_2] &\,=\, \mu_1\beta_2[f_1]+\mu_2\beta_2[f_2] \,=\, 0
\intertext{by an appropriate nontrivial choice of $\mu_1,\mu_2\in\R$.
Likewise, if $M_2$ is nonsingular, then}
\nonumber
   \beta_2[\mu_1f_1+\mu_2f_2] &\,=\, 0 \qquad 
   \text{for every $\mu_1,\mu_2\in\R$}
\intertext{and}
\nonumber
   \beta_1[\mu_1f_1+\mu_2f_2] &\,=\, \mu_1\beta_1[f_1]+\mu_2\beta_1[f_2]
   \,=\, 0\,,
\end{align}
if the nontrivial coefficients $\mu_1$ and $\mu_2$ are chosen appropriately.
Therefore, in either case we can find a probing current 
$f\in\Span\{f_1,f_2\}\setminus\{0\}$,
such that the two leading expansion coefficients $\beta_1=\beta_1[f]$ and 
$\beta_2=\beta_2[f]$ of the corresponding electric potential $u$ are both 
vanishing near the vertex $x_1$.

But then it follows from \req{gamma} that Seo's Assumption~\req{Ass:Seo}
concerning the choice of the two probing currents $f_1$ and $f_2$
is in contradiction to Theorem~\ref{Thm:Seo}. Thus we have proved that
the null space of $\partial\Lambda_{f_1,f_2}(\D)$ is trivial, i.e., 
that $\partial\Lambda_{f_1,f_2}(\D)$ is injective.
\end{proof}

\section{Lipschitz stability for a conductive polygonal inclusion}
\label{Sec:Lipschitz}
The idea of obtaining Lipschitz stability for Seo's inverse problem originates
from the fact that each admissible polygon $\D$ with $n$ vertices
can be described by a $2n$-dimensional vector
\be{x}
   \xx \,=\, [x_1,\dots,x_n] \,\in\,(\R^2)^n
\ee
with the coordinates of its vertices in counterclockwise ordering. 
However, since we have the freedom of choosing any vertex of $\D$ for $x_1$,
this vector is not uniquely specified. 

To account for this problem we introduce the following metric in the set of
all admissible polygons with $n$ vertices 
(which corresponds to a pseudometric in $(\R^2)^n$):
If $\D'$ is a second admissible polygon with $n$ vertices $x_i'$ in 
counterclockwise order, let
\be{metric}
   d(\D,\D') \,=\, \min_{j=0,\dots,n-1} \max_{i=1,\dots,n} |x_{i+j}-x_i'|\,.
\ee
Clearly, $d$ is nonnegative and symmetric; 
it vanishes, if and only if $\D=\D'$. 
To see that $d$ also satisfies the triangle inequality, let $\D''$
be a third admissible polygon with vertices $x_i''$, $i=1,\dots,n$.
Without loss of generality we can assume that the enumeration
of the vertices of $\D$ and $\D'$ is chosen in such a way that
\begin{align}
\nonumber
   d(\D'',\D) \,=\, \max_{i=1,\dots,n} |x_i''-x_i|\,, 
\intertext{and likewise, that}
\nonumber
   d(\D'',\D') \,=\, \max_{i=1,\dots,n} |x_i''-x_i'|\,, 
\end{align}
Then we readily conclude that
\bdmal
   d(\D,\D') 
   &\,\leq\, \max_{i=1,\dots,n} |x_i-x_i'| 
    \,\leq\, \max_{i=1,\dots,n} 
                \bigl(|x_i-x_i''| \,+\, |x_i''-x_i'|\bigr)\\[1ex]
   &\,\leq\, \max_{i=1,\dots,n} |x_i-x_i''| 
             \,+\, \max_{j=1,\dots,n} |x_j''-x_j'| 
    \,=\, d(\D,\D'') \,+\, d(\D'',\D')\,.
\edmal

Take note that $d(\D,\D')$ majorizes the Hausdorff distance 
$d_H(\partial\D,\partial\D')$ between $\partial\D$ and $\partial\D'$. 
For, if $x\in\partial\D$ then there is
some vertex $x_i$ of $\D$ and $c\in[0,1)$, such that
\bdm
   x \,=\, cx_i \,+\, (1-c) x_{i+1}\,;
\edm
assuming further without loss of generality that the minimum in \req{metric}
is attained for $j=0$, then it follows that
\bdm
   \bigl| x \,-\, (cx_i' + (1-c)x_{i+1}')\bigr|
   \,\leq\, c\, |x_i-x_i'| \,+\, (1-c)\,|x_{i+1}-x_{i+1}'|
   \,\leq\, d(\D,\D')\,.
\edm
Since $cx_i'+(1-c)x_{i+1}'\in\partial\D'$ this shows that the distance
between any $x\in\partial\D$ and $\partial\D'$ is at most $d(\D,\D')$,
and hence,
\be{Hausdorff}
   d_H(\partial\D,\partial\D') \,\leq\, d(\D,\D')\,.
\ee

Another useful ingredient to achieve Lipschitz stability is compactness.
Therefore, following Beretta and Francini~\cite{BeFr22}, we define for a
given $\delta>0$ the union $\A_{n,\delta}$ of all admissible polygons
with $n$ vertices, such that
\begin{itemize}
\item[(i)] $|x_i-x|\geq \delta$ for every $i=1,\dots,n$ and every 
$x\in\Gamma_j$ with $j\notin\{i,i+1\}$\,;
\item[(ii)] $\delta\leq \alpha_i \leq 2\pi-\delta$ and 
$|\alpha-\pi|\geq \delta$ for every $i=1,\dots,n$\,;
\item[(iii)] $|x-y|\geq \delta$ for every $x\in\partial\D$ and every 
$y\in\partial\Omega$\,.
\end{itemize}
We denote by $\X_{n,\delta}\subset(\R^2)^n$ the set of coordinate vectors~\req{x},
which describe some $\D\in\A_{n,\delta}$, and we emphasize that $\X_{n,\delta}$ is 
compact. Further, $\X_{n,\delta}\subset\X_{n,\delta/2}$, and when $(\R^2)^n$ 
is equipped with the norm
\be{norm-d}
   \norm{\dd} \,=\, \max_{i=1,\dots,n} |d_i| \qquad 
   \text{for every \ $\dd=[d_1,\dots,d_n]$}\,,
\ee
then there is an open set $\U_{n,\delta}\subset(\R^2)^n$, such that
\bdm
   \X_{n,\delta} \,\subset\, \U_{n,\delta}\,\subset\, \X_{n,\delta/2}\,.
\edm

In $\U_{n,\delta}$ we define 
\be{F}
   F:\begin{cases}
        \phantom{\xx\!\!}\U_{n,\delta} \to\, (L^2_\diamond(\partial\Omega))^2\,,\\
        \phantom{\U_{n,\delta}\!\!\!}\xx \,\mapsto\, \Lambda_{f_1,f_2}(\D)\,,
     \end{cases}
\ee
where $\D\in\A_{n,\delta}$ is the polygon associated with $\xx$.
Further, for the same pair of $\xx$ and $\D$ we introduce
\bdm
   S_{\xx}:\begin{cases}
                  \phantom{\dd\!\!}(\R^2)^n \to\, \S^2\,, \\
                  \phantom{(\R^2)^n\!\!\!}\dd \,\mapsto\, h\,,
               \end{cases}
\edm
where for $\dd=[d_1,\dots,d_n]$ the vector field $h\in\S^2$ is defined by 
piecewise linear interpolation:
\be{interpol}
   h(x_i) \,=\, d_i\,, \qquad i=1,\dots,n\,.
\ee
With the norms introduced in \req{norm-h} and \req{norm-d} $S_\xx$ 
is an isometry.

\begin{proposition}
\label{Prop:Frechet}
The operator $F$ of \req{F} belongs to 
$C^1\bigl(\U_{n,\delta},(L^2_\diamond(\Omega))^2\bigr)$, 
and for $\xx\in\X_{n,\delta}$ and the associated $\D\in\A_{n,\delta}$ 
its derivative is given by
\be{Prop:Frechet}
   F'(\xx)\dd \,=\, \partial\Lambda_{f_1,f_2}(\D)S_\xx\dd\,, \qquad
   \dd\in(\R^2)^n\,.
\ee
\end{proposition}

\begin{proof}
Let $\xx\in\U_{n,\delta}$ and $\dd\in(\R^2)^n$ be so small that $\xx+\dd$ 
also belongs to $\U_{n,\delta}$. Let $\D$ and $\D'$ be the polygons associated
with $\xx$ and $\xx+\dd$, respectively. 
Then, for $F'$ defined as in \req{Prop:Frechet} there holds
\be{Taylorrest-tmp}
\begin{aligned}
   &\norm{F(\xx+\dd) \,-\, F(\xx) \,-\, F'(\xx)\dd}_{(L^2(\partial\Omega))^2} 
    \,=\, \\[1ex]
   &\qquad
    \norm{\Lambda_{f_1,f_2}(\D') \,-\, \Lambda_{f_1,f_2}(\D)
          \,-\, \partial\Lambda_{f_1,f_2}(\D)h}_{(L^2(\partial\Omega))^2}\,,
\end{aligned}
\ee
where $h\in\S^2$ is defined as in \req{interpol}.
As shown in \cite{Hank24a} the right-hand side of \req{Taylorrest-tmp}
can be bounded by $C \norm{h}^2$ for $h$ sufficiently small, i.e.,
for $\dd$ sufficiently small.
The constant $C$ depends on the conductivity $k$
and on the probing currents $f_{1,2}$, 
and also on $\Omega$ and on $\D$,
but this constant can be chosen in such a way that
\be{Taylorrest}
   \norm{F(\xx'+\dd) \,-\, F(\xx') \,-\, F'(\xx')\dd}_{(L^2(\partial\Omega))^2} 
   \,\leq\, C\norm{\dd}^2
\ee
holds true for $\dd$ sufficiently small and all $\xx'$ 
within a certain neighborhood of $\xx$. 
This proves that $F$ is differentiable in $\U_{n,\delta}$. 

To show that $F$ is $C^1$ we let $\xx$ and $\D$ be defined as before,
and we quote from \cite{Hank24a} that
\bdm
   \norm{F'(\xx')}_{{\cal L}((\R^2)^n,(L^2(\partial\Omega))^2)}
   \,=\, \norm{\partial\Lambda_{f_1,f_2}(\D')}_{{\cal L}(\Sp^2,(L^2(\partial\Omega))^2)}
\edm
is uniformly bounded for all $\xx'$ sufficiently close to $\xx$ 
and the associated polygons $\D'$. From this it follows immediately that
\bdm
   \norm{F'(\xx+\dd)}_{{\cal L}((\R^2)^n,(L^2(\partial\Omega))^2)} \,\leq\, L
\edm
for some $L>0$, provided $\dd$ is sufficiently small. Accordingly, 
if $\xx'\in\U_{n,\delta}$ is sufficiently close to $\xx$ then we can estimate
\be{F-Lipschitz}
\begin{aligned}
   \norm{F(\xx')-F(\xx)}_{(L^2(\partial\Omega))^2}
   &\,\leq\, \int_0^1 
               \norm{F'(\xx+t(\xx'-\xx)\bigr)(\xx'-\xx)}_{(L^2(\partial\Omega))^2} 
             \dt\\[1ex]
   &\,\leq\, \int_0^1 L\,\norm{\xx'-\xx} \dt \,=\, L\,\norm{\xx'-\xx}\,.
\end{aligned}
\ee
Now let $\dd$ be an arbitrary unit vector in $(\R^2)^n$. 
Then \req{Taylorrest} and \req{F-Lipschitz} imply that
\bdmal
   & \norm{F'(\xx')\dd - F'(\xx)\dd}_{(L^2(\partial\Omega))^2} 
     \,\leq\, \Bigl\|F'(\xx')\dd
                   \,-\,\frac{1}{t}\bigl(F(\xx'+t\dd)-F(\xx')\bigr)
              \Bigr\|_{(L^2(\partial\Omega))^2}\\[1ex]
   & \qquad \phantom{\,\leq}\
     \,+\, \Bigl\|\frac{1}{t}\bigl(F(\xx+t\dd)-F(\xx)\bigr)
                  \,-\,F'(\xx)\dd\,
           \Bigr\|_{(L^2(\partial\Omega))^2}
     \\[1ex]
   & \qquad \phantom{\,\leq}\
     \,+\, \frac{1}{t}\,\bigl\| F(\xx'+t\dd)-F(\xx+t\dd) 
                        \bigr\|_{(L^2(\partial\Omega))^2}
     \,+\, \frac{1}{t}\,\bigl\| F(\xx')-F(\xx) 
                        \bigr\|_{(L^2(\partial\Omega))^2} \\[2ex]
   & \qquad 
     \,\leq\, 2C t\,\norm{\dd}^2
     \,+\, \frac{2L}{t}\,\bigl\| \xx'-\xx \bigr\|\,,
\edmal
provided that $t>0$ is sufficiently small. 
In particular, for $t=\norm{\xx'-\xx}^{1/2}$ we obtain
\bdm
   \norm{F'(\xx')\dd - F'(\xx)\dd}_{(L^2(\partial\Omega))^2} 
   \,\leq\, (2C+2L) \norm{\xx'-\xx}^{1/2}\,,
\edm
independent of the particular choice of $\dd$. 
This shows that $F\in C^1(\U_{n,\delta},(L^2_\diamond(\Omega))^2)$.
\end{proof}

Now we can formulate our inverse Lipschitz stability result.

\begin{theorem}
\label{Thm:Lipschitz}
Let $\delta>0$ and $\A_{n,\delta}$ be defined as above. Further, let the
two piecewise constant probing currents $f_1,f_2\in L^2_\diamond(\partial\Omega)$
fulfill Seo's Assumption~\req{Ass:Seo}.
Then there is a Lipschitz constant $\ell_{n,\delta}$, depending only on 
$n$, $\delta$, $k$, $\Omega$, and $f_{1,2}$, such that
\be{Thm:Lipschitz}
   d(\D,\D') \,\leq\, \ell_{n,\delta}\,
   \norm{\Lambda_{f_1,f_2}(\D)-\Lambda_{f_1,f_2}(\D')}_{(L^2(\partial\Omega))^2}
\ee
for every pair of polygons $\D,\D'\in\A_{n,\delta}$, where the metric
$d$ is defined in \req{metric}.
\end{theorem}

\begin{proof}
We assume \req{Thm:Lipschitz} to be wrong, i.e., we assume that there exist
two sequences $(\D_m)_m,(\D_m')_m\subset\A_{n,\delta}$ with vertex coordinates
$\xx_m$ and $\xx_m'$ in $(\R^2)^n$, such that
\be{Thm:Lipschitz-wrong}
   d(\D_m,\D_m') \,>\, \eta_m\,\norm{F(\xx_m)-F(\xx_m')}_{(L^2(\partial\Omega))^2}\,,
\ee
where $\eta_m\to\infty$ for $m\to\infty$. 
Since $\X_{n,\delta}$ is compact we can find a subsequence $(m_l)_{l\in\N}$ 
of indices, such that the associated subsequences
$(\xx_{m_l})_l$ and $(\xx'_{m_l})_l$ converge, 
i.e., there exist $\xx,\xx'\in\X_{n,\delta}$ with
\bdm
   \xx_{m_l}\to\xx \quad \text{and} \quad \xx_{m_l}'\to \xx' \qquad
   \text{for $l\to\infty$}\,.
\edm
For ease of notation we consider these subsequences to be the
original sequences that we have started with.
For the two polygons $\D$ and $\D'$ in $\A_{n,\delta}$ associated with $\xx$ and 
$\xx'$, respectively, we then have
\bdm
   d(\D_m,\D) \,\to\, 0 \quad \text{and} \quad d(\D_m',\D')\,\to\,0 \qquad 
   \text{for $m\to \infty$}\,,
\edm
and hence,
\be{DmtoD}
   d(\D_m,\D_m') \,\to\, d(\D,\D') \,<\, \infty\,.
\ee
We further conclude that
\bdm
   F(\xx_m) \,\to\, F(\xx) \quad \text{and} \quad
   F(\xx_m')\,\to\, F(\xx') \qquad \text{for $m\to\infty$}\,,
\edm
because $F$ is differentiable.
Together with \req{Thm:Lipschitz-wrong} and \req{DmtoD} this implies
that $F(\xx)=F(\xx')$. In other words, 
$\Lambda_{f_1,f_2}(\D)=\Lambda_{f_1,f_2}(\D')$,
and by Seo's uniqueness result we necessarily have $\D=\D'$.
Without loss of generality we can assume in the sequel that $\xx=\xx'$; 
otherwise, we reenumerate the vertices of $\D_m'$ for every $m\in\N$ in such 
a way that the enumeration of the vertices of $\D$ and $\D'$ is the same.

For $m$ sufficiently large we now rewrite
\bdmal
   &F(\xx_m) - F(\xx_m') 
    \,=\, \int_0^1 F'\bigl(\xx_m'+t(\xx_m-\xx_m')\bigr)(\xx_m-\xx_m')\dt \\[1ex]
   &\qquad
    \,=\, F'(\xx)(\xx_m-\xx_m') \,+\, 
          \int_0^1\Bigl(F'\bigl(\xx_m'+t(\xx_m-\xx_m')\bigr)
                       -F'(\xx)\Bigr)(\xx_m-\xx_m')\dt\,,
\edmal
and hence, introducing
\bdm
   \dd_m \,=\, \frac{\xx_m-\xx_m'}{\norm{\xx_m-\xx_m'}}
\edm
and
\bdm
   \eps_m \,=\, 
   \int_0^1\bigl\|F'\bigl(\xx_m'+t(\xx_m-\xx_m')\bigr)-F'(\xx)\bigr\|\dt\,,
\edm
we arrive at
\be{Fdiffest}
   \norm{F(\xx_m)-F(\xx_m')}_{(L^2(\partial\Omega))^2}
   \,\geq\, \Bigl(\norm{F'(\xx)\dd_m}_{(L^2(\partial\Omega))^2} \,-\, \eps_m\Bigr)
            \norm{\xx_m-\xx_m'}\,.
\ee
Take note that $\xx_m\neq\xx_m'$ because $d(\D_m,\D_m')>0$ according to
\req{Thm:Lipschitz-wrong}, and that $\eps_m\to 0$ as $m\to\infty$ 
by the continuity of $F'$ and the fact that
\bdm
   \bigl\|\xx_m' + t(\xx_m-\xx_m') - \xx\bigr\| 
   \,\leq\, t\,\norm{\xx_m-\xx} + (1-t)\norm{\xx_m'-\xx}
   \,\to\, 0
\edm
as $m\to\infty$, uniformly for $t\in[0,1]$.
Inserting \req{Fdiffest} into \req{Thm:Lipschitz-wrong} we thus conclude that
\be{Fprimedm}
   \norm{F'(\xx)\dd_m}_{(L^2(\partial\Omega))^2}
   \,<\, \eps_m \,+\, \frac{1}{\eta_m}\,
         \frac{d(\D_m,\D_m')}{\|\xx_m-\xx_m'\|} 
\ee
for $m$ sufficiently large. 
Since $\xx_m-\xx_m'\to 0$ and the vertices of $\D_m$ 
(and of $\D_m'$, respectively) are at least $\delta$ apart by requirement~(i)
in the definition of $\A_{n,\delta}$, we conclude from \req{metric} that
\bdm
   d(\D_m,\D_m') \,=\, \norm{\xx_m-\xx_m'}\,, \qquad
   \text{if \ $\|\xx_m-\xx_m'\| \,<\,\delta/2$}\,,
\edm
and hence, the right-hand side of \req{Fprimedm} goes to zero as $m\to\infty$.

Since $(\dd_m)_m$ consists of unit vectors from $(\R^2)^n$ we can find a 
convergent subsequence -- again denoted by $(\dd_m)_m$ -- with
\bdm
   \dd_m \,\to\,\dd\,, \qquad m\to\infty\,,
\edm
where the limit $\dd=[d_1,\dots,d_n]\in(\R^2)^n$ also has norm one. 
It thus follows from \req{Fprimedm} that
\bdm
   0 \,=\, F'(\xx)\dd \,=\, \partial\Lambda_{f_1,f_2}(\D)h\,,
\edm
where $0\neq h\in\S^2$ is given by \req{interpol}. But this violates our
injectivity result in Theorem~\ref{Thm:Fprime-injective}, and hence, we
have the desired contradition to \req{Thm:Lipschitz-wrong}.
\end{proof}

As we have mentioned in the introduction the idea of this proof 
is borrowed from \cite{BaVe06,Bour13}. We rearranged the argumentation, 
though, to deal with the difficulties that the operator $F$ of \req{F} 
fails to be injective and its domain $\X_{n,\delta}$ is not convex.

For ease of completeness we also state the following Lipschitz stability
result for admissible polygons with \emph{at most} $N$ vertices.

\begin{corollary}
\label{Cor:Lipschitz}
Let $\delta>0$ and $\B_{N,\delta}=\bigcup_{n=3}^N\A_{n,\delta}$ for some $N\geq 3$. 
Further, let the two piecewise constant probing currents 
$f_1,f_2\in L_\diamond^2(\partial\Omega)$ fulfill Seo's Assumption~\req{Ass:Seo}. 
Then there is a Lipschitz constant $L_{N,\delta}$, depending only on 
$N$, $\delta$, $k$, $\Omega$, and $f_{1,2}$, such that
\bdm
   d_H(\partial\D,\partial\D') \,\leq\, L_{N,\delta} \,
   \norm{\Lambda_{f_1,f_2}(\D) - \Lambda_{f_1,f_2}(\D')}_{L^2(\partial\Omega)^2}
\edm
for every pair of polygons $\D,\D'\in\B_{N,\delta}$, where $d_H$ denotes the
Hausdorff metric.
\end{corollary}

\begin{proof}
As in the previous proof we assume to the contrary that there are sequences
$(\D_m)_m,(\D_m')_m\subset\B_{N,\delta}$ with
\be{Ass:Cor:Lipschitz}
   d_H(\partial\D_m,\partial\D_m') \,>\, \eta_m\,
   \norm{\Lambda_{f_1,f_2}(\D_m) - \Lambda_{f_1,f_2}(\D_m')}_{L^2(\partial\Omega)^2}\,,
\ee
where $\eta_m\to\infty$ as $m\to\infty$. Then there are infinitely many
indices $m_l$, $l\in\N$, and two natural numbers $n,n'\in\{3,\dots,N\}$ such
that all $\D_{m_l}$ are $n$-gons and all $\D_{m_l}'$ are $n'$-gons.
Again we assume that these two subsequences have been the original ones,
and as in the proof of Theorem~\ref{Thm:Lipschitz} we can further assume 
without loss of generality that the corresponding vectors $\xx_m\in\X_{n,\delta}$
and $\xx_m'\in\X_{n',\delta}$ with the coordinates of the vertices of
$\D_m$ and $\D_m'$, respectively, converge. If we denote the polygons
corresponding to the two limit vectors
by $\D\in\A_{n,\delta}$ and $\D'\in\A_{n',\delta}$ this implies that
\bdmal
   &\norm{\Lambda_{f_1,f_2}(\D) - \Lambda_{f_1,f_2}(\D')}_{L^2(\partial\Omega)^2}
   \,=\, \lim_{m\to\infty} 
         \norm{\Lambda_{f_1,f_2}(\D_m)
               - \Lambda_{f_1,f_2}(\D_m')}_{L^2(\partial\Omega)^2}\\[1ex]
   &\qquad
    \,\leq\, \lim_{m\to\infty} \frac{1}{\eta_m}\, 
             d_H(\partial\D_m,\partial\D_m') \,=\, 0\,,
\edmal
because $d_H(\partial\D_m,\partial\D_m')\to d_H(\partial\D,\partial\D')$.
From Seo's uniqueness result therefore follows that $\D=\D'$, 
and in particular, that $n=n'$. 
We can thus apply Theorem~\ref{Thm:Lipschitz} to conclude that
\be{Cor:Lipschitz}
   d(\D_m,\D_m') \,\leq\, \ell_{n,\delta}\,
   \norm{\Lambda_{f_1,f_2}(\D_m) - \Lambda_{f_1,f_2}(\D_m')}_{L^2(\partial\Omega)^2}
\ee
for all $m\in\N$. Combined with \req{Ass:Cor:Lipschitz} and \req{Hausdorff} 
this implies that the sequence $(\eta_m)_m$ is bounded, 
which is the desired contradiction.
\end{proof}

\section{The case of an insulating polygonal inclusion}
\label{Sec:insulating_case}
So far we have assumed that the conductivity $k$ of the inclusion is positive.
The limiting case $k=0$ corresponds to an insulating inclusion, which is
more adequately modeled by the boundary value problem
\be{insulating-pde}
   \Delta u=0 \quad \text{in $\Omega\setminus\overline\D$}\,, \qquad
   \frac{\partial}{\partial\nu} u = 0 \quad \text{on $\partial\D$}\,,\qquad
   \frac{\partial}{\partial\nu} u = f \quad \text{on $\partial\Omega$}\,.
\ee
If $\D$ is an admissible polygon, then \req{insulating-pde} has a unique
weak solution in
\bdm
   u\in H^1_\diamond(\Omega\setminus\overline\D) \,=\,
   \Bigl\{\,u \in H^1(\Omega\setminus\overline\D) \,:\,
            \int_{\partial\Omega}u\ds = 0\,\Bigr\} \,,
\edm
and we now use
\bdm
   \Lambda_f \,:\, \D \,\mapsto\, u|_{\partial\Omega} 
   \,\in\, L^2_\diamond(\partial\Omega)
\edm
as the associated forward operator. Recall from the introduction that
$\D$ is uniquely determined by $\Lambda_f(\D)$, provided that $f\neq 0$,
cf.~\cite{BeVe98}.

Again, $\Lambda_f$ turns out to be shape differentiable for polygonal
inclusions, cf.~\cite{Hank24a}: For $h\in \S^2$ the shape derivative
$\partial\Lambda_f(\D)h$ is given by the trace on $\partial\Omega$ of
the solution $u'$ of the Neumann boundary value problem
\be{Hank24}
\begin{array}{c}
   \Delta u' \,=\, 0 \quad \text{in $\Omega\setminus\overline\D$}\,, \qquad
   {\displaystyle \int_{\partial\Omega} u'\ds \,=\, 0\,,}
   \\[3ex]
   \dfrac{\partial}{\partial\nu} u'
   \,=\, \dfrac{\partial}{\partial\tau}
         \Bigl((h\cdot\nu)\dfrac{\partial}{\partial\tau}u\Bigr)
         \quad \text{on $\partial\D$}\,, \qquad
   \dfrac{\partial}{\partial\nu}u' \,=\, 0 \quad \text{on $\partial\Omega$}\,.
\end{array}
\ee
Note that $u$ is defined only in the exterior of $\D$, but using
the reflection principle it can be extended as a harmonic function across 
any edge of $\D$ into the interior of $\D$. Therefore $u$ is smooth on 
$\partial\D$, except for the vertices of $\D$.

We now establish injectivity of $\partial\Lambda_f(\D)$.

\begin{theorem}
\label{Thm:Fprime-injective0}
Let $\D$ be an admissible polygon and
$f\in L^2_\diamond(\partial\Omega)$ be a nontrivial probing current.
Then $\partial\Lambda_f(\D)$ is injective.
\end{theorem}

\begin{proof}
Again we assume to the contrary that there exists some nontrivial $h\in\S^2$,
for which $\partial\Lambda_f(\D)h=0$. Further, as in the proof of
Theorem~\ref{Thm:Fprime-injective}, we stipulate without loss of generality 
that 
\be{varphi-local0}
   h(x)\cdot\nu(x) \,=\, 1 + b \,|x-x_1|\,, \qquad x\in\Gamma_2\,,
\ee
for some $b\in\R$. Throughout we use the same notation for the vertices, 
edges, and interior angles of $\D$ as in Section~\ref{Sec:forward}, 
and introduce the same local coordinates~\req{localcoordinates} near 
the vertices.

Let $u'$ be the solution of \req{Hank24} for this particular $h$.
Since $u'|_{\partial\Omega} = \partial\Lambda_f(\D)h = 0$ it follows from
Holmgren's theorem that $u'$ is vanishing in all of 
$\Omega\setminus\overline\D$. Therefore,
the Neumann boundary condition for $u'$ on $\partial\D$ in \req{Hank24}
must be zero, i.e., there is some constant $c\in\R$, such that
\bdm
   (h\cdot\nu)\, \frac{\partial}{\partial\tau} u \,=\, c \qquad
\text{on $\Gamma_2$}\,.
\edm
It thus follows from \req{varphi-local0} that
\be{u-Gamma-representation}  
   u(x) \,=\, 
   \begin{cases}
      u(x_1)\,+\,c\,|x-x_1|\,, & b=0\,, \\[1ex]
      u(x_1)\,+\,\dfrac{c}{b}\log\bigl|1+b\,|x-x_1|\bigr| \,, & b\neq 0\,,
   \end{cases}
   \qquad x\in\Gamma_2\,.
\ee   
Together with the insulating boundary condition $\partial u/\partial\nu = 0$ 
on $\Gamma_2$ this yields a Cauchy problem for $u$ 
with real analytic data on $\Gamma_2$, which has a unique solution 
in $\Omega\setminus\overline\D$. Consider first the case $b=0$. 
In this case the solution of the Cauchy problem is obviously given by
\bdm
   u(x) \,=\, u(x_1) + c\, r\cos\theta\,, \qquad
   0<r<r_0\,, \ \alpha_1\leq\theta\leq 2\pi\,,
\edm
in the local coordinate system~\req{localcoordinates}, and hence,
\bdm
   \left.\frac{\partial}{\partial\nu} u(x)\right|_{\Gamma_1}\!\!
   \,=\, \left.\!\frac{1}{r}\frac{\partial}{\partial\theta} u(x)
         \right|_{\theta=\alpha_1}\!\!
   \,=\, -c\sin\alpha_1\,, \qquad 0<r<r_0\,.
\edm
But this must be zero in the insulating case, proving that $c=0$,
because $\alpha_1\notin\{0,\pi,2\pi\}$.

In the case when $b\neq 0$ in \req{varphi-local0} we have 
\begin{align*}
   u(x) &\,=\, u(x_1) \,+\, 
               \frac{c}{b}\sum_{j=1}^\infty \frac{(-1)^{j+1}}{j}\, b^j|x-x_1|^j
\intertext{according to \req{u-Gamma-representation} for $x$ near $x_1$ on 
$\Gamma_2$, and the solution of the Cauchy problem is given by}
   u(x) &\,=\, u(x_1) \,+\, 
               c\sum_{j=1}^\infty \frac{(-1)^{j+1}}{j} b^{j-1} r^j \cos j\theta
\end{align*}
in the local coordinate system~\req{localcoordinates} with $0<r=|x-x_1|<r_0$
and $\alpha_1\leq\theta\leq 2\pi$, and hence,
\bdm
   0 \,=\, \left.\frac{\partial}{\partial\nu} u(x)\right|_{\Gamma_1} \!\!
     \,=\, c\sum_{j=0}^\infty (-1)^{j+1} \thsp b^j r^j\sin (j+1)\alpha_1
\edm
for $0<|x-x_0|<r_0$. Again this is only possible if $c=0$.

In either case we have shown that $c=0$ in
\req{u-Gamma-representation}, which implies that $u$ is constant 
in $\Omega\setminus\overline\D$.
But this is a contradiction to the assumption 
that $f$ is a nontrivial boundary current, i.e., that
\bdm
   f \,=\, \left.\!\frac{\partial}{\partial\nu} u\right|_{\partial\Omega} 
     \,\neq\, 0\,.
\edm
We have thus established the injectivity of $\partial\Lambda_f(\D)$.
\end{proof}

We can now proceed as in Section~\ref{Sec:Lipschitz}, define
for $\delta>0$ and $n\geq 3$ the operator
\bdm
   F:\begin{cases}
        \phantom{\xx\!\!}\U_{n,\delta} \to\, L^2_\diamond(\partial\Omega)\,,\\
        \phantom{\U_{n,\delta}\!\!\!}\xx \,\mapsto\, \Lambda_f(\D)\,,
     \end{cases}
\edm
and use results from \cite{Hank24a} to show that $F$ belongs to 
$C^1(\U_{n,\delta},L^2_\diamond(\partial\Omega))$.
Without any change of proof we thus get the following result.

\begin{theorem}
\label{Thm:k=0}
Let $\delta>0$ and $n,N\geq 3$. Furthermore, let $\A_{n,\delta}$ and 
$\B_{N,\delta}$ be defined as in Section~\ref{Sec:Lipschitz},
and let $f\in L_\diamond^2(\partial\Omega)$ be a nontrivial probing current.
Then there are positive Lipschitz constants $\ell'_{n,\delta}$ and $L'_{N,\delta}$, 
depending only on $n$ (resp.\ $N$), $\delta$, $\Omega$, and $f$, such that
\bdm
   d(\D,\D') \,\leq\,
   \ell'_{n,\delta}\,\norm{\Lambda_f(\D)-\Lambda_f(\D')}_{L^2(\partial\Omega)}
\edm
for every pair of polygons $\D,\D'\in\A_{n,\delta}$, and
\bdm
   d_H(\partial\D,\partial\D') \,\leq\,
   L'_{N,\delta}\,\norm{\Lambda_f(\D)-\Lambda_f(\D')}_{L^2(\partial\Omega)}
\edm
for every pair of polygons $\D,\D'\in\B_{N,\delta}$.
Here, $d$ is the metric defined in \req{metric} and $d_H$ is the
Hausdorff metric.
\end{theorem}


\end{document}